\documentclass[11pt,reqno]{amsart}
\usepackage{amsmath,amsthm,amssymb,mathrsfs,stmaryrd,color, amsfonts}
\usepackage[all]{xy}
\usepackage{url}
\usepackage{geometry}
\numberwithin{equation}{section}

\usepackage[utf8]{inputenc}
\usepackage[T1]{fontenc}

\usepackage{extarrows}
\usepackage{mathtools}
\usepackage{enumerate}
\usepackage[colorlinks=true,hyperindex, linkcolor=blue, pagebackref=false, citecolor=blue,pdfpagelabels]{hyperref}
\usepackage[capitalize]{cleveref}

\newtheorem{theorem}{Theorem}[section]
\newtheorem*{theorem*}{Theorem}
\newtheorem*{definition*}{Definition}

\newtheorem{theo}[theorem]{Theorem}
\newtheorem*{theo*}{Theorem}
\newtheorem{lem}[theorem]{Lemma}
\newtheorem{prop}[theorem]{Proposition}
\newtheorem{cor}[theorem]{Corollary}

\theoremstyle{definition}

\newtheorem{rem}[theorem]{Remark}
\newtheorem{expl}[theorem]{Example}
\newtheorem{defn}[theorem]{Definition}

\DeclareMathOperator{\ch}{ch}
\DeclareMathOperator{\Coh}{Coh}

\DeclareMathOperator{\Ext}{Ext}
\DeclareMathOperator{\GL}{GL}
\DeclareMathOperator{\K}{K}
\DeclareMathOperator{\Hom}{Hom}
\DeclareMathOperator{\id}{id}
\DeclareMathOperator{\im}{im}
\DeclareMathOperator{\Jac}{Jac}
\DeclareMathOperator{\mult}{mult}
\DeclareMathOperator{\M}{M}

\DeclareMathOperator{\NS}{NS}
\DeclareMathOperator{\Pic}{Pic}

\DeclareMathOperator{\rk}{rk}
\DeclareMathOperator{\SL}{SL}
\DeclareMathOperator{\SM}{SM}
\DeclareMathOperator{\Spec}{Spec}
\DeclareMathOperator{\Supp}{Supp}
\DeclareMathOperator{\Sym}{Sym}
\DeclareMathOperator{\td}{td}

\DeclareMathOperator{\sheafHom}{\mathcal{H}\it{om}}

\DeclareMathOperator{\sheafExt}{\mathcal{E}\it{xt}}
\DeclareMathOperator{\sheafTor}{\mathcal{T}\it{or}}

\renewcommand{\epsilon}{\varepsilon}

\newcommand{\Ecal}{\mathcal{E}}
\newcommand{\Fcal}{\mathcal{F}}
\newcommand{\Gcal}{\mathcal{G}}

\newcommand{\Lcal}{\mathcal{L}}
\newcommand{\Mcal}{\mathcal{M}}
\newcommand{\Ncal}{\mathcal{N}}
\newcommand{\Ocal}{\mathcal{O}}
\newcommand{\Pcal}{\mathcal{P}}

\newcommand{\Tcal}{\mathcal{T}}
\newcommand{\Ucal}{\mathcal{U}}

\newcommand{\Vcal}{\mathcal{V}}
\newcommand{\Wcal}{\mathcal{W}}

\newcommand{\Q}{\mathbb{Q}}
\newcommand{\Z}{\mathbb{Z}}
\newcommand{\C}{\mathbb{C}}

\newcommand{\Abb}{\mathbb{A}}

\newcommand{\Pbb}{\mathbb{P}}

\newcommand{\Wbar}{\overline{W}}

\newcommand{\sss}{\scriptscriptstyle}

\usepackage{color}
\newcounter{commentcounter}

\def\?{\ 
{\bf\color{red}???}\ 
\immediate\write16{}
\immediate\write16{Warning: There was still a question mark . . . }
\immediate\write16{}}

\begin{document}

\title[Nilpotent cone in the Mukai system for rank 2 and genus 2]{The nilpotent cone in the Mukai system for rank 2 and genus 2}
	\author[I.~Hellmann]{Isabell Hellmann}
	\address{Mathematisches Institut, Universit\"at Bonn, Endenicher Allee 60, 53115 Bonn, Germany}
	\email{igb@math.uni-bonn.de}
	
	\begin{abstract}\noindent
	We study the nilpotent cone in the Mukai system for rank two and genus two. We compute the degrees and multiplicities of its irreducible components and describe their cohomology classes.
	\end{abstract}
	
	\maketitle
	{\let\thefootnote\relax\footnotetext{The author is supported by the SFB/TR 45 `Periods, Moduli Spaces and Arithmetic of Algebraic varieties' of the DFG (German Research Foundation) and the Bonn International Graduate School.}\marginpar{}}
\section{Introduction} 
Let $(S,H)$ be a polarized K3 surface of genus $g$ and fix two coprime integers $n\geq 1$ and $s$. The moduli space $M=M_S(v)$ of $H$-Gieseker stable coherent sheaves of Mukai vector $v=(0,nH,s)$ is a smooth Hyperk\"ahler variety of dimension $2(n^2(g-1)+1)$. A point in $M$ corresponds to a stable sheaf $\Ecal$ on $S$ such that $\Ecal$ is pure of dimension one with support in the linear system $|nH|$. Taking (Fitting) supports defines a Lagrangian fibration
\[ f \colon M \longrightarrow |nH| \cong \Pbb^{n^2(g-1)+1} \]
known as the \emph{Mukai system} \cite{Beau}, \cite{Mu}. Over a general point in $|nH|$ which corresponds to a smooth curve $D \subset S$ the fibers of $f$ are smooth abelian varieties isomorphic to $\Pic^{\delta}(D)$, where $\delta = s-n^2(1-g)$. So, $M$ can also be viewed as a relative compactified Jacobian $\Jac^{\delta}(\mathcal{C}/|nH|) \rightarrow |nH|$ associated to the universal curve $\mathcal{C} \rightarrow |nH|$.

The Mukai system is of special interest because of its relation to the classical and widely studied \emph{Hitchin system}, see \cite{Ha} for a survey. Let $C$ be a smooth curve of genus $g$. A Higgs bundle on $C$ is a pair $(\Ecal,\phi)$ consisting of a vector bundle $\Ecal$ on $C$ and a morphism $\phi \colon \Ecal \rightarrow \Ecal \otimes \omega_C$ called Higgs field. The moduli space $M_{\rm Higgs}(n,d)$ of Higgs bundles of rank $n$ and degree $d$ is a smooth and quasi-projective symplectic variety. Sending $(\Ecal,\phi)$ to the coefficients of its characteristic polynomial $\chi(\phi)$ defines a proper Lagrangian fibration
\[ \chi \colon M_{\rm Higgs}(n,d) \longrightarrow \bigoplus_{i = 1}^{n}H^0(C,\omega_C^i). \]
It is equivariant with respect to the $\C^*$-action that is given by scaling the Higgs field on $M_{\rm Higgs}(n,d)$ and by multiplication with $t^i$ in the corresponding summand on the base. As a corollary the topology of $M_{\rm Higgs}(n,d)$ is controlled by the fiber over the origin. This fiber
\[ N \coloneqq \chi^{-1}(0) = \{(\Ecal,\phi)\in M_{\rm Higgs}(n,d)\ |\ \phi\ \text{is nilpotent} \}\]
is called the \emph{nilpotent cone}. In the late \!'80s Beauville, Narasimhan, and Ramanan discovered a beautiful interpretation of the space of Higgs bundles \cite{BNR}. They showed that a Higgs bundle $(\Ecal,\phi)$ with characteristic polynomial $s$ corresponds to a pure sheaf of rank one on a so called spectral curve $C_s \subset |\omega_C|$ inside the total space of the canonical bundle. The curve $C_s$ is defined in terms of $s =\chi(\phi)$ and is 
linearly equivalent to $nC$, the $n$-th order thickening of the zero section $C \subset |\omega_C|$. This idea was taken up by Donagi, Ein, and Lazarsfeld in \cite{DEL}: The space $M_{\rm Higgs}(n,d)$ appears as a moduli space of sheaves on $|\omega_C|$ that are supported on curves in the linear system $|nC|$. Consequently, $M_{\rm Higgs}(n,d)$ has a natural compactification $\overline{M}_{\rm Higgs}(n,d)$ given by a moduli space of sheaves on the projective surface $S_0 = \Pbb(\omega_C\oplus \Ocal_C)$ with polarization $H_0=\Ocal_{S_0}(C)$. The Hitchin map extends to $\overline{M}_{\rm Higgs}(n,d) \rightarrow |nH_0| \cong \Pbb(\oplus_{i=0}^nH^0(\omega_C^i))$ and is nothing but the support map; the nilpotent cone is the fiber over the point $nC \in |nH_0|$. However, $\overline{M}_{\rm Higgs}(n,d)$ cannot admit a symplectic structure as it is covered by rational curves. At this point the Mukai system enters the picture. If $S$ is a K3 surface that contains the curve $C$ as a hyperplane section, one can degenerate $(S,H)$ to $(S_0,H_0)$ and consequently the Mukai system $M_S(v) \rightarrow |nH|$ with $v=(0,nH,d+n(1-g))$ degenerates to the compactified Hitchin system \cite[\S 1]{DEL}. From our perspective, this is a powerful approach to studying the Hitchin system. For instance, in a recent paper \cite{CMS}, de Cataldo, Maulik and Shen prove the P=W conjecture for $g=2$ by means of the corresponding specialization map on cohomology.

In this note, we study the geometry of the nilpotent cone in the Mukai system, which is defined in parallel to the Hitchin system
\[ N_C \coloneqq f^{-1}(nC),\]
for some curve $C \in |H|$. We will fix the invariants $n=2$ and $g=2$ and the Mukai vector $v=(0,2H,-1)$. In this case, the nilpotent cone has two irreducible components
\[ (N_C)_{\rm red} = N_0 \cup N_1,\]
where the first component is isomorphic to the moduli space $M_C(2,1)$ of stable vector bundles of rank two and degree one on $C$ and the second component is the closure of $N_C\setminus N_0$. Both components are Lagrangian subvarieties of $M=M_S(v)$. Understood with their reduced structure, $N_0$ is smooth and the singularities of $N_1$ are contained in $N_0 \cap N_1$. However, both components occur with multiplicities. Our first result is the computation of the multiplicities of the components as well as their degrees. Here, the degree is meant with respect to a distinguished ample class $u_1 \in H^2(M,\Z)$, see Definition \ref{u_1}.

\begin{theo}\label{Thm1}
The degrees of the two components of the nilpotent cone $N_C$ are given by
\begin{equation*}
\begin{array}{lcr}
\deg_{u_1}N_0 = 5 \cdot 2^9 & \text{and} & \deg_{u_1}N_1 = 5^2\cdot 2^{11}
\end{array}
\end{equation*}
and their multiplicities are
\begin{equation*}
\begin{array}{lcr}
\mult_{N_C}N_0 = 2^3 & \text{and} & \mult_{N_C}N_1 = 2.
\end{array}
\end{equation*}
Moreover, any fiber $F$ of the Mukai system has degree $5 \cdot 3 \cdot 2^{13}$.
\end{theo}

As the smooth locus of every component with its reduced structure deforms from the Mukai to the Hitchin system, the multiplicities and degrees must coincide. Here, indeed, the same multiplicities can be found in \cite[Propositions (34) and (35)]{Th} and \cite[Proposition 6]{Hitchin}, whereas, up to our knowledge, the degrees have not been determined in the literature.

Our second result is a description of the cohomology classes $[N_0]$ and $[N_1] \in H^{10}(M,\Z)$. The projective moduli spaces of stable sheaves on K3 surfaces are known to be deformation equivalent to Hilbert schemes of points. In our case, $M$ is actually birational to $S^{[5]}$ \cite[Lemma 3.2.7]{CRS}. In particular, there is an isomorphism 
$ H^*(M,\Z) \cong H^*(S^{[5]},\Z)$.
The cohomology ring of $S^{[5]}$ is well understood, e.g.\ \cite[\S 4]{Le} and the references therein. Recall that for any Hyperk\"ahler variety $X$ of dimension $2n$ there is an embedding $S^iH^2(X,\Q) \hookrightarrow H^{2i}(X,\Q)$ for all $i\leq 2n$ \cite[Theorem 1.7]{V}.
\begin{theo}\label{Thm2}
The classes $[N_0]$ and $[N_1] \in H^{10}(M,\Q)$ are linearly independent and span a totally isotropic subspace of $H^{10}(M,\Q)$ with respect to the intersection pairing. They are given by
\[ [N_0] = \frac{1}{48} [F] + \beta\ \text{and}\   [N_1] = \frac{5}{12} [F] - 4 \beta,
\]
where $[F]$ is the class of a general fiber of the Mukai system and $0 \neq \beta \in (S^5H^2(M,\Q))^\bot$ satisfies $\beta^2 = 0$. As $\deg_{u_1}\beta = 0$, the class $\beta$ is not effective.
\end{theo}

The structure of the paper is as follows:
In Section \ref{section Mukai} we introduce the Mukai system. In Section \ref{section nilpotent cone} we describe the irreducible components of the nilpotent cone following \cite[\S 3]{DEL}, where it is shown that any point $[\Ecal]\in N\setminus N_0$ fits into an extension of the form
\[ 0 \rightarrow \Lcal(x)\otimes \omega_C^{-1} \longrightarrow \Ecal \longrightarrow \Lcal \rightarrow 0 ,\]
where $\Lcal \in \Pic^1(C)$ is a line bundle and $x \in C$ a point. We specify a space $W \rightarrow {\Pic^1(C)} \times C$ parameterizing such extensions, and a compactification $\Wbar$ of $W$ that comes with a birational map $\nu \colon \Wbar \rightarrow N_1$. In the Hitchin case, this idea originates from \cite{Th}.\\
In Section \ref{section degrees} we prove Theorem \ref{Thm1}. The proof relies on the functorial properties of the defintion of $u_1$ via the determinant line bundle construction, see Section \ref{construction of u_1}. It allows us to relate $u_1\!\left|_{F}\right.$ and $u_1\!\left|_{N_0}\right.$ with the (generalized) theta divisor on $F = f^{-1}(D) \cong \Pic^3(D)$ for $D \in |nC|$ smooth and $M_C(2,1)$, respectively, see Propositions \ref{deg of general fiber} and \ref{vector bundle comp}. For $[N_1]$ the degree computation is achieved by determining $\nu^*u_1 \in H^2(\Wbar,\Z)$. Finally, the multiplicities are infered from knowing the degrees.
The last Section \ref{section classes} is devoted to the proof of Theorem \ref{Thm2}. It uses our previous results.\\

All schemes are of finite type over $k = \C$. In the entire paper, $S$ is a K3 surface that contains a smooth curve $C$ as a hyperplane section and $H = \Ocal_S(C) \in \NS(S)$.

\subsubsection*{Acknowledgements}
I am very grateful to Daniel Huybrechts for his invaluable support. I wish to thank Thorsten Beckmann, Norbert Hoffmann, Hsueh-Yung Lin, Georg Oberdieck, Giulia Sacc\`a and Andrey Soldatenkov for helpful discussions and Tony Pantev for providing the secret manuscript.

\section{The Mukai system}\label{section Mukai}
In this section, we give a brief recollection on moduli spaces of sheaves on K3 surfaces and define the Mukai system. First recall that the Mukai vector induces an isomorphism
\[v \colon \K(S)_{\rm num} \xrightarrow\sim H^*_{\rm alg}(S,\Z) = H^0(S,\Z) \oplus \NS(S) \oplus H^4(S,\Z).\]
It is given by
$v(\Ecal) \coloneqq \ch(\Ecal)\sqrt{\td(S)} = (\rk(\Ecal),c_1(\Ecal),\chi(\Ecal)-\rk(\Ecal)).$

We write $M_{S}(v)$ for the moduli space of pure, $H$-Gieseker stable sheaves on $S$ with Mukai vector $v$. For the sake of readability, we suppress the polarization from our notation, even though it may change the isomorphism class of the moduli space. If $v$ is primitive and positive and $H$ is $v$-generic then $M_{S}(v)$ is an irreducible holomorphic symplectic manifold of dimension $\langle v,v \rangle +2$, which is deformation equivalent to the Hilbert scheme of $\tfrac{1}{2}\langle v,v \rangle +1$ points on $S$ \cite[Theorem 10.3.1]{K3}. Here, $\langle\ ,\ \rangle$ is the Mukai pairing given by $\langle(r,c,s),(r',c',s') \rangle = cc' -rs'-r's$.

Consider the Mukai vector
\[v = (0,nH,s) \in H^*_{\rm alg}(S,\Z),\]
and assume that $v$ is primitive. A pure sheaf $\Fcal$ of Mukai vector $v$ has one-dimensional support, first Chern class $nH$ and Euler characteristic $s$. In particular, $\Fcal$ admits a length one resolution
$ 0 \rightarrow \Vcal \xlongrightarrow f \tilde{\Vcal} \longrightarrow \Fcal$
by two vector bundles of the same rank $r$ \cite[\S 1.1]{HL}. We define the \emph{(Fitting) support} of $\Fcal$ to be
\[ \Supp(\Fcal) \coloneqq V(\det f) \subset S \]
the vanishing scheme of the induced morphism $\det f= \wedge^r f \colon \wedge^r\Vcal \rightarrow \wedge^r\tilde{\Vcal}$, for any resolution $0 \rightarrow \Vcal \xrightarrow f \tilde{\Vcal}$ of $\Fcal$ as above.
This definition is well-defined, i.e.\ independent of the chosen resolution \cite[Definition 20.4]{Ei}.

\begin{expl}
Let $i \colon C \hookrightarrow S$ be an integral curve and $\Ecal$ a vector bundle of rank $n$ on $C$. Then
\[ \Supp(i_*\Ecal) = nC \]
is the $n$-th order thickening of $C$ in $S$.
\end{expl}

By definition, $\Supp(\Fcal)$ is linearly equivalent to $c_1(\Fcal) = nH$ and $\Supp(\Fcal)$ contains the usual support defined by the annihilator of $\Fcal$. Moreover, the reduced locus $\Supp(\Fcal)_{\rm red}$ is the set-theoretic support of $\Fcal$. The advantage of the above definition is, that it behaves well in families and thus induces a morphism \cite[\S 2.2]{LP}
\[ f \colon M_S(v) \longrightarrow  |nH| \cong \Pbb^{\tilde{g}},\ \ [\Ecal] \mapsto \Supp(\Ecal).\]
Here, $\tilde{g}= n^2(g-1)+1$. Moreover, $M_S(v)$ is irreducible holomorphic symplectic of dimension $n^2H^2+2= 2\tilde{g}$ and hence, by Matsushita's result \cite[Corollary 1]{Ma} this morphism is a Lagrangian fibration, also called \emph{Mukai system} (an explicit proof can be found in \cite[Lemma 1.3]{DEL}). 

\section{The nilpotent cone for \texorpdfstring{$n=2$}{} and \texorpdfstring{$g=2$}{}}\label{section nilpotent cone}
We now specialize to the case that $n = 2$ and $s = 3-2g$ with $g=2$, i.e.\ we fix the Mukai vector
\[v = (0,2H,-1).\]
In particular, a stable vector bundle of rank two and degree one on a smooth curve $C \in |H|$ defines a point in $M \coloneqq M_S(v)$. We have $\dim M = 8g-6 = 10$ and $M$ is birational to the Hilbert scheme $S^{[5]}$ of five points on $S$.

Taking (Fitting) supports defines a Lagrangian fibration
\[ f \colon M \longrightarrow |2H| \cong \Pbb^{5}. \]
Let $\Delta \subset \Sigma \subset |2H|$ be the subloci corresponding to non-reduced and non-integral curves, respectively. Under the natural morphism $|H|\times |H| \rightarrow |2H|$, we have $\Sigma \cong \Sym^2|H|$ and $\Delta \cong |H|$ identifies with the diagonal. As in \cite[Proposition 3.7.1]{CRS} we distinguish three cases:
\begin{equation*}\label{fibertype}
f^{-1}(x)\
\begin{cases}
\text{is reduced and irreducible} & \text{if}\ x\in |2H|\setminus\Sigma\\
\text{is reduced and has two irreducible components} & \text{if}\ x\in \Sigma\setminus\Delta \\
\text{has two irreducible components with multiplicities} & \text{if}\ x\in \Delta.
\end{cases}
\end{equation*}


We will study fibers of the third type, namely
\[ N_C \coloneqq f^{-1}(2C), \]
where $C \in |H|$. In analogy with the Hitchin system, we call $N_C$ \emph{nilpotent cone}.
\begin{rem}\label{reduction to smooth C}
According to \cite[Proposition 3.7.4]{CRS}, we have a decomposition into irreducible components
\[f^{-1}(\Delta)_{\rm red} = N^\Delta_0 \cup N^\Delta_1,\]
such that $N^\Delta_i$ is flat over $\Delta$, and $f^{-1}(C)\cap N^\Delta_i$ is irreducible for $i=0,1$. 
Consequently, the multiplicities and degrees of the irreducible components of $N_C$ do not depend on $C \in |H|$.
\end{rem}

For the rest of the paper, we fix a smooth curve $C \in |H|$ and write $N$ instead of $N_C$. We will now identify the irreducible components of $N$ following the ideas of \cite{DEL}.

\subsection{Pointwise description of the nilpotent cone \texorpdfstring{$N$}{N}}
Let $[\Ecal] \in N$ and consider its restriction $\Ecal\!\left|_C\right.$ to $C$. There are two cases, either $\Ecal\!\left|_C\right.$ is a stable rank two vector bundle on $C$ or $\Ecal\!\left|_C\right.$ has rank one. By dimension reasons, the sheaves of the first kind contribute an irreducible component $N_0$ of $N$ isomorphic to the moduli space $\M_C(2,1)$ of stable rank two and degree one vector bundles on $C$. In the second case, $\Ecal\!\left|_C\right. \cong \Lcal \oplus \Ocal_D$, where the first factor $\Lcal \coloneqq \Ecal\!\left|_C\right./\text{torsion}$ is a line bundle on $C$ and $D\subset C$ is an effective divisor. We set
\[E_1 \coloneqq N \setminus N_0\]
with reduced structure.

\begin{lem}\label{temp1}
Let $[\Ecal] \in E_1$ and write $\Ecal\!\left|_C\right. = \Lcal \oplus \Ocal_D$. There is a short exact sequence of $\Ocal_S$-modules
\begin{equation}\label{key extension}
0 \rightarrow i_*(\Lcal(D) \otimes \omega_C^{-1}) \longrightarrow \Ecal \longrightarrow i_*\Lcal \rightarrow 0.
\end{equation}
Moreover, $k \coloneqq \deg\Lcal =1$ and $d \coloneqq \deg D = 2g-2k-1 =1$.
\end{lem}

\begin{proof}
Noting that $\omega_C^{-1}$ is the conormal bundle of $C$ in $S$, it is straightforward to obtain the sequence  \eqref{key extension}. Let us prove the numerical restrictions. From \eqref{key extension}  we have
\[1 + 2(1-g) = \chi(\Ecal) = \chi(\Lcal(D)\otimes \omega_C^{-1}) + \chi(\Lcal) = 2k+d-(2g-2) + 2(1-g).\]
Thus $d = 2g -2k -1$ and we find $k \leq g-1$. On the other hand, $\Ecal$ is stable, so the reduced Hilbert polynomials of $\Ecal$ and $\Lcal$ satisfy $p(\Ecal,t) < p(\Lcal,t)$, which amounts to
\[  \tfrac{1}{2}{(1+2(1-g))} < {k+1-g}\]
or equivalently  $k\geq 1$.
\end{proof}

\begin{rem}
For $n=2$ and arbitrary genus $g$, one has $\deg\Lcal \in \{1,\ldots,g-1\}$ and a decomposition into locally closed subsets
$ N_{\rm red}=N_0 \sqcup E_1 \sqcup \ldots \sqcup E_{g-1} $
corresponding to the degree of $\Lcal$. In fact, $N_0$ and the closures of $E_k$ are the irreducible components of $N$.
\end{rem}

We conclude that every point in $E_1$ defines a class in $\Ext^1_S(i_*\Lcal,i_*(\Lcal(x)\otimes \omega_C^{-1}))$ for some point $x \in C$ and some line bundle $\Lcal \in \Pic^1(C)$. Conversely, an extension class in $\Ext^1_S(i_*\Lcal,i_*(\Lcal(x)\otimes \omega_C^{-1}))$  defines a point in $E_1$ if and only if its middle term is stable and has the point $x$ as support of its torsion part when restricted to $C$, i.e.\ if it is not pushed forward from $C$. It turns out that all such extensions are stable.

\begin{lem}\label{stability}
Consider a coherent sheaf $\Ecal$ on $S$ that is given as an extension
\[ 0 \rightarrow i_*\Lcal \rightarrow \Ecal \rightarrow i_*\Lcal' \rightarrow 0, \]
where $\Lcal'$ and $\Lcal$ are line bundles on $C$ of degree $k$ and $1-k$, respectively, with $k \geq 1$. Moreover, assume that $\Ecal$ itself does not admit the structure of an $\Ocal_C$-module. Then $\Ecal$ is $H$-Gieseker stable.
\end{lem}
\begin{proof}
We have to prove $p(\Ecal,t) < p(\Mcal,t)$ or, equivalently,
$\tfrac{\chi(\Ecal)}{c_1(\Ecal).H} < \tfrac{\chi(\Mcal)}{c_1(\Mcal).H}$
for every surjection $\Ecal \twoheadrightarrow \Mcal$.
We can assume that $\Supp(\Mcal)=C$ and $\Mcal =  i_*\Mcal'$, where $\Mcal'$ is a line bundle on $C$. Then because $\Ecal\!\left|_C\right. \cong \Lcal' \oplus \Tcal$ for some torsion sheaf $\Tcal$, we find
\[ \sheafHom_{\Ocal_S}(\Ecal,i_*\Mcal') \cong \sheafHom_{\Ocal_C}(\Ecal\!\left|_C\right.,\Mcal') \cong \sheafHom_{\Ocal_C}(\Lcal',\Mcal') \]
and thus $i_*\Lcal' \xrightarrow\sim \Mcal$.
\end{proof}

\begin{cor}\label{points in E1}
The closed points of $E_1$ are in bijection with the following set
\begin{equation*}
\bigsqcup\limits_{\stackrel{\Lcal \in \Pic^1(C)}{x \in C}} \Pbb(\Ext^1_S(i_*\Lcal,i_*(\Lcal(x)\otimes \omega_C^{-1})))\setminus \Pbb(\Ext^1_C(\Lcal,\Lcal(x)\otimes \omega_C^{-1})),
\end{equation*}
i.e.\ with extension classes $[v] \in \Pbb(\Ext^1_S(i_*\Lcal,i_*(\Lcal(x)\otimes \omega_C^{-1})))$ such that $v$ is not pushed forward from $C$. Here, $\Lcal$ varies over all line bundles on $C$ with $\deg\Lcal = 1$, and $x$ varies over all points in $C$.
The bijection is established by Lemma \ref{temp1}.
\end{cor}

From Proposition \ref{lemma1} we have a short exact sequence
\[ 0 \rightarrow \Ext^1_C(\Lcal,\Lcal(x)\otimes \omega_C^{-1}) \rightarrow \Ext^1_S(i_*\Lcal,i_*(\Lcal(x)\otimes \omega_C^{-1})) \xrightarrow{\rho_{\Lcal,x}} H^0(C,\Ocal_C(x)) \rightarrow 0 \]
and the following interpretation of the morphism $\rho_{\Lcal,x}$ modulo a scalar factor. If $\Ecal$ is the middle term of a representing sequence of $v \in \Ext^1_S(i_*\Lcal,i_*(\Lcal(x)\otimes \omega_C^{-1}))$, then 
\[\Ecal\!\left|_C\right. \cong \Lcal \oplus \Ocal_{V(\rho_{\Lcal,x}(v))}.\]
Hence, another way to phrase the corollary above is by fixing for every $x \in C$ a defining section $s_x \in H^0(C,\Ocal_C(x))$ as follows. Let $\Delta \hookrightarrow C \times C$ be the diagonal,
yielding a section $s_{\Delta} \in H^0(C \times C,\Ocal(\Delta))$. For every $x \in C$, we set $s_x = s_{\Delta}\!\left|_{\{x\}\times C}\right.$. Then we can write
\begin{equation}\label{unionoffibers}
\begin{array}{ccc}
\text{points of}\ E_1 & \stackrel{1:1}{\longleftrightarrow} & \bigsqcup\limits_{\stackrel{\Lcal \in \Pic^1(C)}{x \in C}}\{ v \in \Ext^1_S(i_*\Lcal,i_*(\Lcal(x)\otimes \omega_C^{-1}))\ |\ \rho_{\Lcal,x}(v) = s_x \}.
\end{array}
\end{equation}

\subsection{Extension spaces}
So far, we have given a pointwise description of the nilpotent cone. Next, we will identify its irreducible components and their scheme structures. This section is a technical parenthesis in this direction. The reader may like to skip it.

Let $S$ be a smooth projective surface and $i \colon C \hookrightarrow S$ a smooth curve with normal bundle $\Ncal_{C/S} \cong \Ocal_C(C)$. Let $T$ be any scheme and let $\Fcal$ and $\Fcal'$ be two vector bundles on $T\times C$ considered as families of vector bundles on $C$. Denote by $\pi\colon T \times S \rightarrow T$ and $\pi_C\colon T \times C\rightarrow T$ the projections. For a morphism $f\colon X \rightarrow Y$, we write $\sheafExt_f$ instead of $Rf_*R\sheafHom$.

\begin{prop}\label{lemma1}
There is an exact sequenc of $\Ocal_T$-modules
\begin{equation}\label{relativeext}
0 \rightarrow \sheafExt^1_{\pi_C}(\Fcal',\Fcal)\rightarrow \sheafExt^1_{\pi}((\id\times i)_*\Fcal',(\id\times i)_*\Fcal) \xrightarrow{\rho} \sheafExt^0_{\pi_C}(\Fcal'\boxtimes \Ocal_C(-C),\Fcal) \rightarrow 0
\end{equation}
as well as for every $t \in T$ an exact sequence of vector spaces
\begin{equation}\label{fiberext}
0 \rightarrow \Ext^1_{C}(\Fcal'_t,\Fcal_t) \xlongrightarrow{\xi} \Ext^1_{S}(i_*\Fcal'_t,i_*\Fcal_t) \xlongrightarrow{\rho_t} \Ext^0_C(\Fcal'_t \otimes \Ocal_C(-C),\Fcal_t) \rightarrow 0.
\end{equation}
\end{prop}

Note that the fibers of \eqref{relativeext} must, in general, not coincide with \eqref{fiberext}, see Lemma \ref{bcforpoints}.

\begin{proof}
Apply ${R\pi_C}_* R\sheafHom(\;\;,\Fcal)$ or $R\Hom(\;\;,\Fcal_t)$, to the exact triangle
\begin{equation*}
\Fcal' \boxtimes \Ocal_C(-C) [1] \rightarrow L(\id\times i)^*(\id\times i)_*\Fcal' \rightarrow \Fcal' \xrightarrow{[1]}
\end{equation*}
in $D^b(T\times C)$ (see \cite[Corollary 11.4]{FM}) or its counterpart
in $D^b(C)$, respectively, and consider the induced cohomology sequence.
\end{proof} 

We can explicitly describe the morphism $\rho_t$ in the sequence \eqref{fiberext}. Represent $v \in \Ext^1_{S}(i_*\Fcal_t',i_*\Fcal_t)$ by
$0 \rightarrow i_*\Fcal_t \rightarrow \Ecal \rightarrow i_*\Fcal_t' \rightarrow 0.$
Restriction to $C$ yields
\[ \ldots \rightarrow \Fcal_t' \otimes \Ocal_C(-C) \xrightarrow{\delta(v)} \Fcal_t \rightarrow \Ecal\!\!\left|_C\right. \rightarrow \Fcal_t' \rightarrow 0, \]
where we inserted
$\sheafTor_1^{\Ocal_S}(i_*\Fcal_t',i_*\Ocal_C)\cong \Fcal_t' \otimes_{\Ocal_C} \sheafTor_1^{\Ocal_S}(i_*\Ocal_C,i_*\Ocal_C)  = \Fcal_t' \otimes \Ocal_C(-C).$
This gives a well-defined, linear map
\[\delta\colon \Ext^1_{S}(i_*\Fcal_t',i_*\Fcal_t) \rightarrow \Ext^0_C(\Fcal_t' \otimes \Ocal_C(-C),\Fcal_t).\]
As $\im \xi = \ker \delta$, it follows by dimension reasons, that $\delta$ has to be surjective. So, $\rho_t = \delta$ up to post-composition with an isomorphism of $\Ext^0_C(\Fcal_t' \otimes \Ocal_C(-C),\Fcal_t)$.

\begin{lem}\label{bcforpoints}
For every $t \in T$ there is a commutative diagram of short exact sequences
\begin{equation*}
\xymatrix{
\sheafExt^1_{\pi_C}(\Fcal',\Fcal)(t) \ar[r] \ar[d]^\cong & \sheafExt^1_{\pi}((\id\times i)_*\Fcal',(\id\times i)_*\Fcal)(t) \ar[r]^{\rho(t)} \ar[d] & \sheafExt^0_{\pi_C}(\Fcal'\boxtimes \Ocal_C(-C),\Fcal)(t)\ar[d] \\
\Ext^1_{C}(\Fcal_t',\Fcal_t) \ar[r] &\Ext^1_{S}(i_*\Fcal'_t,i_*\Fcal_t) \ar[r]^-{\rho_t} & \Ext^0_C(\Fcal_t' \otimes \Ocal_C(-C),\Fcal_t),
}
\end{equation*}
where the first vertical arrow is an isomorphism. If $\Ext^0_C(\Fcal_t' \otimes \Ocal_C(-C),\Fcal_t)$ has constant dimension for all $t \in T$ all vertical arrows are isomorphisms.
\end{lem}

\begin{proof}
The vertical  morphisms are the usual functorial base change morphisms. 
The lower line is \eqref{fiberext} and hence also exact on the left. The first vertical arrow is an isomorphism because $\Ext^2_C(\Fcal'_t,\Fcal_t)=0$. Consequently, also the upper line is exact on the left.
\end{proof}

\subsection{Irreducible components of \texorpdfstring{$N$}{}}
In this section, we show that $E_1$ is irreducible and has the same dimension as $N$. Therefore its closure
\[ N_1 \coloneqq \overline{E}_1 \subset N_{\rm red}, \]
with reduced structure is an irreducible component of $N$. For the proof, we need some more notation. Let $\Pcal_1$ be a Poincar\'e line bundle on ${\Pic^1(C)} \times C$ and $\Delta \subset C \times C$ the diagonal. Set $T \coloneqq {\Pic^1}(C) \times C$ and on $T$ define the following sheaves
\begin{align*}
\Vcal &\coloneqq R^1{p_{\sss 12}}_*(p_{\sss 23}^*\Ocal(\Delta)\otimes p_3^*\omega_C^{-1}),\\
\Wcal &\coloneqq R^1{p_{\sss 12}}R\sheafHom((\id\times i)_*p_{\sss 13}^*\Pcal_1,(\id\times i)_*(p_{\sss 13}^*\Pcal_1\otimes p_{\sss 23}^*\Ocal(\Delta) \otimes p_{\sss 3}^*\omega_C^{-1} )\ \text{and}\\
\Ucal &\coloneqq {p_{\sss 12}}_*p_{\sss 23}^*\Ocal(\Delta),
\end{align*}
where $p_{ij}$ are the appropriate projections from ${\Pic^1(C)} \times C \times C$. Considering the fiber dimensions, we see that $\Vcal$ and $\Ucal$ are vector bundles of rank 2 and 1, respectively. In fact, $s_\Delta$ induces an isomorphism
$s_\Delta \colon \Ocal_T \xlongrightarrow\sim \Ucal = {p_{\sss 12}}_*{p_{\sss 23}}^*\Ocal(\Delta).$
Moreover, by Proposition \ref{lemma1} they fit into a short exact sequence
\begin{equation}\label{ses}
0 \rightarrow \Vcal \longrightarrow \Wcal \xlongrightarrow{\rho} \Ocal_T \rightarrow 0. 
\end{equation}
Consequently, also $\Wcal$ is a vector bundle and $\rho$ induces a map of geometric vector bundles
\[|\rho| \colon |\Wcal| = \underline{\smash{\Spec}}_T(\Sym^\bullet\Wcal^\vee)  \longrightarrow T \times \Abb^1. \]
We set
\[  W \coloneqq |\rho|^{-1}(T\times \{1\}) \]
with the projection $\tau \colon W \rightarrow T$. We retain some immediate consequences of the construction.
\begin{enumerate}[\rm (i)]
\item $W$ is a principal homogeneous space under $|\Vcal|$. In particular, it is an affine bundle over $T$.
\item Let $t=(\Lcal,x) \in T$. Then by Lemma \ref{bcforpoints} we have
\[ W_t = \tau^{-1}(t) \cong \Pbb(\Ext^1_S(i_*\Lcal,i_*(\Lcal(x)\otimes \omega_C^{-1})))\setminus \Pbb(\Ext^1_C(\Lcal,\Lcal(x)\otimes \omega_C^{-1})).\]
\item $\dim W = 5$.
\item $W$ is compactified by the projective bundle $\Wbar \coloneqq \Pbb(\Wcal)$ with boundary isomorphic to $\Pbb(\Vcal)$, i.e.
\[ \Wbar = W \cup \Pbb(\Vcal). \]
\end{enumerate}

\begin{rem} Actually, $\Vcal \cong p_2^*(\omega_C^{-1}\oplus \omega_C^{-1})$ and hence $\Pbb(\Vcal) \cong \Pbb^1 \times {\Pic^1(C)} \times C$.
\end{rem}

Next, we relate $E_1$ and $\Wbar$. Recall that $N_1 \coloneqq \overline{E}_1 \subset N_{\rm red}$.
We keep all the notations from the previous section, and
\begin{equation*}
\xymatrix{\Wbar  \times C \ar[d]^-{\tau_C} \ar@{^(->}[r] & \Wbar \times S \ar[r]^-{\pi'} \ar[d]^-{\tau_S} & \Wbar \coloneqq \Pbb(\Wcal) \ar[d]^-{\tau} \\
T \times C \ar@{^(->}[r] & T\times S \ar[r]^-{\pi} & T.
}
\end{equation*}
\begin{prop}\label{nu is there and birational}
There exists a `universal' extension represented by
 \[  0 \rightarrow \tau_S^*(\id\times i)_*(\Pcal_1 \boxtimes \Ocal(\Delta) \boxtimes \omega_C^{-1})\boxtimes \Ocal_\tau(1) \rightarrow \Gcal_{\rm univ} \rightarrow \tau_S^*(\id\times i)_*p_{\sss 13}^*\Pcal_1 \rightarrow 0,\]
 such that $\Gcal_{\rm univ} \in \Coh(\overline{W} \times S)$ defines a birational morphism
\[ \nu \colon \overline{W} \longrightarrow N_1.\]
In particular, $N_{\rm red} = N_0 \cup N_1$ is a decomposition into irreducible components.
\end{prop}

\begin{proof}
We set $\Fcal \coloneqq \Pcal_1 \boxtimes \Ocal(\Delta) \boxtimes \omega_C^{-1}  $ and  $\Fcal' \coloneqq p_{\sss 13}^*\Pcal_1$. We are looking for a `universal' extension, i.e.\ for
\[ v_{\sss \rm univ} \in \Ext^1_{\Wbar \times S}(\tau_S^*(\id\times i)_*\Fcal',\tau_S^*(\id\times i)_*\Fcal \otimes \pi'^*\Ocal_\tau(1)),\]
such that for $w \in W \subset \overline{W}$ the restriction of $v_{\sss \rm univ}$ to $\{w\} \times S$ is the extension corresponding to $w \in W_{\tau(w)} \subset \Ext^1_S(i_*\Fcal'_{\tau(w)},i_*\Fcal_{\tau(w)})$.\\
By definition $\Wcal = R^1{\pi}_*R\sheafHom((\id\times i)_*\Fcal',(\id\times i)_*\Fcal)$. Hence, there is a base change map
$\tau^*\Wcal \rightarrow R^1{\pi'}_*L\tau_S^*R\sheafHom((\id\times i)_*\Fcal',(\id\times i)_*\Fcal)$.
We get
\begin{equation}\label{bla}
\begin{aligned}
H^0(\Wbar, \tau^*\Wcal \otimes \Ocal_\tau(1))
 &\rightarrow H^0(\Wbar,R^1{\pi'}_*L\tau_S^*R\sheafHom((\id\times i)_*\Fcal',(\id\times i)_*\Fcal) \otimes \Ocal_{\tau}(1))\\
 &\xleftarrow{\sim} H^1(\Wbar \times S,L\tau_S^*R\sheafHom((\id\times i)_*\Fcal',(\id\times i)_*\Fcal) \otimes \pi'^* \Ocal_{\tau}(1)))\\
  &= \Ext^1_{\Wbar \times S}(\tau_S^*(\id\times i)_*\Fcal',\tau_S^*(\id\times i)_*\Fcal \otimes \pi'^*\Ocal_\tau(1)),
  \end{aligned}
  \end{equation}
 
where the indicated isomorphism comes from the Leray spectral sequence. It is an isomorphism, because \begin{equation*}
R^0{\pi'}_*L\tau_S^*R\sheafHom((\id\times i)_*\Fcal',(\id\times i)_*\Fcal)
\cong R^0{\pi_C'}_*L\tau_C^*R\sheafHom(L(\id\times i)^*(\id\times i)_*\Fcal',\Fcal)=  0,
\end{equation*}
where $\pi_C \colon T \times C \rightarrow T$. The last equality follows from the long exact sequence
\begin{multline*}
\ldots \rightarrow  0 =  R^{0}{\pi_C'}_*L\tau_C^*R\sheafHom(\Fcal',\Fcal) \rightarrow R^0{\pi_C'}_*L\tau_C^*R\sheafHom(L(\id\times i)^*(\id\times i)_*\Fcal',\Fcal)\\
\rightarrow 0 =R^0{\pi_C'}_*L\tau_C^*R\sheafHom(\Fcal'\boxtimes \Ocal_C(-C)[1],\Fcal)\rightarrow \ldots.
\end{multline*}
Finally, we consider the universal surjection 
as an element in $H^0(\Wbar, \tau^*\Wcal \otimes \Ocal_\tau(1)))$ and take its image under \eqref{bla}. This produces the desired extension.

By construction, $\Gcal_{\rm univ} \in \Coh(\overline{W} \times S)$ defines a morphism
$\nu \colon \overline{W} \rightarrow N_1 \subset M$
which restricts to a bijection $W \rightarrow E_1$ (see Corollary \ref{points in E1} and \eqref{unionoffibers}). Moreover, the boundary $\overline{W}\setminus W = \Pbb(\Vcal)$ maps to $N_1\setminus E_1 = N_0 \cap N_1$.  Note that by degree reasons an extension on $C$ that is of the form
$0 \rightarrow \Lcal \rightarrow \Ecal \rightarrow \Lcal' \rightarrow 0,$
where $\deg\Lcal' = 1$ and $\deg\Lcal = 0$ is stable or split. However, the split extensions do not occur in $\Pbb(\Vcal)$. Hence, $\nu$ is everywhere defined.
\end{proof}

\begin{rem}
One can show that $\nu \colon W \rightarrow E_1$ is actually an isomorphism of varieties. Moreover, $\nu\colon \overline{W} \rightarrow N_1$ is finite and hence a normalization map. Its tangent map is analyzed in \cite[Proposition 7.5]{DPS} and provides a characterization of the singularities of $N_1$. 
\end{rem}

\section{Proof of Theorem \ref{Thm1}}\label{section degrees}
We will now prove Theorem \ref{Thm1} stating the degrees and multiplicities of the two components $N_0$ and $N_1$ of $f^{-1}(2C)$, where $C \in |H|$. In Remark \ref{reduction to smooth C}, we saw already that it is enough to consider a fixed smooth curve $C$. All degrees will be computed with respect to a particular ample class $u_1 \in H^2(M,\Z)$ which we construct in Section \ref{construction of u_1}. We set
\[ d_i = \deg_{u_1}(N_i) \coloneqq \int_M [N_i]u_1^5\]
for $i= 0,1$, where by abuse of notation $[N_i] \in H^{10}(M,\Z)$ is the Poincar\'e dual of the fundamental homology class $[N_i] \in H_{10}(M,\Z)$. We want to prove 
\begin{equation*}
\begin{array}{lcr}
d_0 = 5 \cdot 2^9 & \text{and} & d_1 = 5^2\cdot 2^{11}.
\end{array}
\end{equation*}
The multiplicity is defined as follows. Let $\eta_i$ be the generic point of $N_i$. Then
\[ m_i =  \mult_N N_i \coloneqq \lg_{\Ocal_{N,\eta_i}}\Ocal_{N_i,\eta_i} = \lg_{\Ocal_{N_i,\eta_i}}\Ocal_{N_i,\eta_i}. \]
In particular, we have an equality $[F] = m_0[N_0]+m_1[N_1] \in H^{10}(M,\Z)$ for any fiber $F$. We want to prove
\begin{equation*}
\begin{array}{lcr}
m_0 = 2^3 & \text{and} & m_1 = 2.
\end{array}
\end{equation*}
Consequently, $\deg_{u_1}(F) = 5 \cdot 2^{12}  + 5^2 \cdot 2^{12} =5 \cdot 3 \cdot 2^{13}$.

\begin{proof}[Proof of the multiplicities knowing all the degrees] Let $F \subset M$ be a smooth fiber. Then, we have
$\deg F = m_0d_0 + m_1d_1$ and hence
$5\cdot 3 \cdot 2^{13} = m_0 \cdot 5 \cdot 2^9 + m_1 \cdot 5^2 \cdot 2^{11}.$
The only possible solutions are $(m_0,m_1) = (28,1)$ or $(m_0,m_1) = (8,2)$. However, by \cite[Proposition 4.11]{CK}
\[\dim T_{[\Ecal]}N= \dim\Ext^1_{2C}(\Ecal,\Ecal) =  \dim N + 1\ \text{for all}\ [\Ecal] \in E_1.\]
So $N_1$ is not reduced and the first solution is ruled out.
\end{proof}

\subsection{Construction of the ample class \texorpdfstring{$u_1$}{}}\label{construction of u_1}
We use the determinant line bundle construction in order to produce an ample class on the moduli space $M$.

Let $X$ and $T$ be two projective varieties and assume that $X$ is smooth. Let $p\colon T\times X \rightarrow T$ and $q\colon T \times X \rightarrow X$ denote the two projections. For any $\Wcal \in \Coh(X\times T)$ flat over $T$, we define
$\lambda_{\Wcal} \colon \K(X)_{\rm num} \rightarrow H^2(T,\Z)$
to be the following composition (\cite[Lemma 8.1.2]{HL})
\[ \K(X)_{\rm num} \xrightarrow{q^*} \K^0(T \times X)_{\rm num} \xrightarrow{\cdot [\Wcal]} \K^0(T \times X)_{\rm num} \xrightarrow{Rp_*} \K^0(T)_{\rm num} \xrightarrow{\det} \NS(T) \subset H^2(T,\Z). \]

We will take advantage of the functorial properties of this definition. These are $f^*\lambda_{\Wcal} = \lambda_{(f \times \id)^*\Wcal}$ for any morphism $f\colon T' \rightarrow T$ and $\lambda_{(\id\times i)_*\Wcal}(x) = \lambda_\Wcal(Li^*x)$ for all $x \in \K(X)_{\rm num}$ if $i\colon Y \hookrightarrow X$ is the inclusion of a closed, smooth subscheme and $\Wcal \in \Coh(T \times Y)$.

The construction is of especially interesting if $X=\M_T(c)$ is a fine moduli space, that parametrizes coherent sheaves of class $c$ on $T$. Let $\Ecal_{\rm univ}$ be a universal sheaf on $\M_T(c)\times T$, then
$\lambda_{\Ecal_{\rm univ} \otimes p^*\Mcal}(x) = \lambda_{\Ecal_{\rm univ}}(x) + \chi(c\cdot x) c_1(\Mcal) $
for all $\Mcal \in \Pic(\M_T(c))$. Hence,
\[ \lambda_{\M_T(c)} \coloneqq \lambda_{\Ecal_{\rm univ}} \colon c^\bot \longrightarrow \NS(\M_T(c)) \]
is well-defined and does not depend on the choice of universal sheaf. Here,
\[c^\bot = \{x \in \K(T)_{\rm num}\ |\ \chi(x\cdot c) = 0 \}.\]

\begin{expl}\label{thetas}
Let $C$ be a smooth curve of any genus $g \geq 0$. Then
\[ (\rk,\deg) \colon \K(C)_{\rm num} \xlongrightarrow\sim \Z \oplus \Z. \]
Fix $n \geq 1 $ and $d \in \Z$ coprime and let $c = (n,d)\in  \K(C)_{\rm num}$. Then $\M_C(c) = \M_C(n,d)$ is the moduli space of stable vector bundles of rank $n$ and degree $d$ on $C$ and we find 
$c ^\bot = \langle (-n, d+n(1-g)\rangle$. The generalized Theta divisor can be defined by
\[ \Theta_{\M_C(n,d)} \coloneqq \lambda_{\M_C(n,d)}(-n,d+n(1-g)),\]
see \cite[Th\'eor\`eme D]{DN}. A special case is $\M_C(1,k) = \Pic^k(C)$. In this case, $c^\bot = \langle (-1,k + 1-g) \rangle$ and
\[\Theta_k \coloneqq \lambda_{\Pic^k(C)}(-1,k+1-g)\]
is the class of the canonical Theta divisor in $\Pic^k(C)$.
\end{expl}

\begin{rem}\label{pullback of theta}
Denote by $\SM_C(n,d)$ the moduli space of vector bundles with fixed determinant, i.e.\ a fiber of $\det \colon \M_C(n,d) \rightarrow \Pic(C)$ and by $\Theta_{\SM_C(n,d)}$ the restriction of $\Theta_{\M_C(n,d)}$ to $\SM_C(n,d)$. Tensor product defines an \'etale map
\[ h \colon \SM_C(n,d) \times \Pic^0(C) \longrightarrow \M_C(n,d), \]
of degree $n^{2g}$. Using \cite[Corollary 6]{DT}, we find the following relation if $(n,d)$ are coprime
\begin{equation}\label{formula for theta}
h^* \Theta_{\M_C(n,d)} = p_1^*\Theta_{\SM_C(n,d)} + n^2p_2^*\Theta_0.
\end{equation}
\end{rem}

\begin{lem} \label{lambda poincare}
Let $C$ be a smooth curve of genus $g$ and $\Pcal$ a Poincar\'e line bundle on $\Pic^k(C) \times C$. Then
\[ \lambda_{\Pcal} \colon \K(C)_{\rm num} \rightarrow H^2(\Pic^k(C),\Z) \]
 is given by
\[(r,d) \mapsto (d+(k+1-g)r)\mu-r\Theta_k,\]
where $p_1^*\mu =c_1^{2,0}(\Pcal) \in H^2(\Pic^k(C)\times C,\Z)$ is the $(2,0)$ K\"unneth component of $c_1(\Pcal)$.
\end{lem}

By tensoring with a suitable line bundle on $\Pic^k(C)$ one can assume that $c_1^{2,0}(\Pcal) = 0$.

\begin{proof}
Let us abbreviate $\Pic^k(C)$ to $\Pic^k$.
We decompose $c_1(\Pcal) =c^{2,0} + c^{1,1} + c^{0,2}$ into its K\"unneth components and write $c^{2,0}= p^*\mu$ for some $\mu \in H^2(\Pic^k,\Z)$ . Then by \cite[VIII \S 2]{ACGH} the class $\gamma = c^{1,1}$ satisfies $\gamma^2 = -2\rho p^*\Theta_k$. Moreover, by definition, $c^{0,2} = k\rho$, where $\rho$ is the pullback of the class of a point on $C$. Together, $c_1(\Pcal) = p^*\mu + \gamma + k\rho$ and
\[
\ch(\Pcal) = 1 + p^*\mu + \gamma + k\rho + \rho p^*(k\mu - \Theta_k).\]
Now, let $x = (r,d) \in \K(C)_{\rm num}$. The Grothendieck--Riemann--Roch theorem gives
\begin{align*}
\ch (Rp_*(\Pcal \otimes q^*x)) &= p_*(\ch(\Pcal\otimes q^*x)\td(\Pic^k \times C)) = p_*(\ch(\Pcal)\ch(q^*x)q^*\td(C)) \\
& =p_*(\ch(\Pcal)(r+((1-g)r+d)\rho))\\
&= kr + (1-g)r + d +(kr + (1-g)r) + d)\mu - r\Theta_k.
\end{align*}

In particular, $\lambda_{\Pcal}(x) = c_1(Rp_*(\Pcal \otimes q^* x)) = ((k+1-g)r + d)\mu -r\Theta_k$.
\end{proof}

Let us come back to our original situation, i.e.\ $(S,H)$ is a polarized K3 surface and $M= M_S(v)$ parametrizes $H$-stable sheaves with Mukai vector $v = (0,2H,-1)$ or equivalently, with 
Chern character
$v' =  \tfrac{v}{\sqrt{\td(S)}} =(0,2H,-1).$ As $v$ and $v'$ coincide, we will notationally not distinguish between them anymore.

\begin{defn}
For all $s \in \Z$ we define 
\[ l_s \coloneqq \lambda_ {M}((-4,-H,s)) \in H^2(M,\Z). \]
Note that ${v}^\bot = \{ (2c.H,c,s)\ |\ c\in \NS(S), s\in \Z\}$, and in particular, $(-4,-H,s) \in v^\bot$.
\end{defn}

\begin{rem}
The value of $s$ does not have any relevance for computations. However, with the results of \cite{BM}, it can be proven that $l_s$ is ample for $s \gg 0$. Or, if $\Pic(S) \cong \Z\cdot H$ this can be seen as follows: By \eqref{Mukai homo} $\NS(M)_\Q$ is then two-dimensional with basis $\{ l_0, u_0 \coloneqq \lambda_M(0,0,1)\}$. So $l_s = l_0 + s\cdot u_0$ must hit the ample cone either if $s \ll 0$ or if $s\gg 0$. Proposition \ref{deg of general fiber} shows, that it is $s \gg 0$.
\end{rem}

For everything what follows, we fix $s_0 \gg 0$ such that $l_{s_0}$ is ample.
\begin{defn}\label{u_1}
We set
\[u_1 \coloneqq l_{s_0}.\]
\end{defn}

\subsection{Degree of a general fiber}
We compute the degree of a general fiber.
\begin{prop}\label{deg of general fiber}
Let $D \in |2H|$ be a smooth curve and let $F \coloneqq f^{-1}(D)$ be the corresponding fiber. Let $u= \lambda_{M}(x)$ with $x=(2c.H,c,s) \in v^\bot$. Then
\[ u\!\left|_{F}\right. = -2c.H \cdot \Theta_3, \]
where $\Theta_3 \in H^2(\Pic^3(D),\Z)$ is the class of the Theta divisor. In particular, we have
\[ \deg_{u_1}F = 5!\cdot 2^{10}.\]
\end{prop}

\begin{proof}
Let $i \colon D \hookrightarrow S$ be the inclusion. The inclusion $\Pic^3(D) \cong F \hookrightarrow M$ is defined by $(\id\times i)_*\Pcal_3$, where $\Pcal_3$ is a Poincar\'e line bundle on $\Pic^3(D) \times D$. Hence,
\[ u|_{\Pic^3(D)} = \lambda_{(\id\times i)_*\Pcal_3}(x) = \lambda_{\Pcal_3}(Li^*x). \]
Now,
$Li^* \colon K(S)_{\rm num} \rightarrow K(D)_{\rm num}  \cong \Z^{\oplus 2}$
maps $(r,c,s)$ to $(r,c.D)$ and thus $Li^*x$ to $2c.H(1,1)$, whereas by definition
$\theta_3 = \lambda_{\Pcal_3}(-1,-1).$
Finally,
\[\deg_{u_1}F = \int_{\Pic^3(D)}(4\Theta_3)^5 = 2^{10} \cdot 5!.\]
\end{proof}

\begin{rem}
One can also prove the above result using the Beauville--Bogomolov form $(\ ,\ )_{BB}$ on $H^2(M,\Z)$.  Let $u_0 = f^*c_1(\Ocal(1)) \in H^2(M,\Z)$. Then $[F] = u_0^5 \in H^{10}(M,\Z)$ and
\[ \deg_{u_1}(F) = \int_M u_0^5u_1^5 = 5!\cdot (u_0,u_1)_{BB}^5,  \]
where we use that $(u_0,u_0)_{BB}=0$ and that $M$ is birational to $S^{[5]}$ in order to determine the correct  Fujiki constant. One verifies that $u_0= \lambda_M((0,0,1))$ whereas, by definition, $u_1 = \lambda_M(-4,-H,s_0)$ with $s_0 \gg 0$. Finally, $\lambda_M$ induces an isomorphism (see \cite[Theorem 6.2.15]{HL})
\begin{equation}\label{Mukai homo}
\lambda_M \colon v^\bot  \xlongrightarrow\sim \NS(M)
\end{equation}
such that
$(\lambda_{M}(r,c,s), \lambda_{M}(r',c',s'))_{BB} = \langle (r,c,s),(r',c',s')\rangle +2rr',$
i.e.\ the Beauville--Bogomolov form corresponds to the Mukai paring after correct identifications.
This gives $(u_0,u_1)_{BB} =4$.
\end{rem}

\subsection{Degree of the vector bundle component \texorpdfstring{$N_0$}{}}
Next, we deal with the component $N_0$, which  is isomorphic to $M_C(2,1)$.
\begin{prop}\label{vector bundle comp}
Let $x= \lambda_{M}(u)$ with $u=(2c.H,c,s) \in v^\bot$. Then
\[ x\!\left|_{N_0}\right. = -c.H\Theta, \]
where $\Theta \in H^2(N_0,\Z)$ is the the generalized Theta divisor.
In particular,
$$u_1\!\left|_{N_0}\right. = 2 \Theta,$$
and given $x_i = \lambda_{M}(2c_i.H,c_i,s_i)$ for $i=1,\ldots,5$, we find
\[ \int_M x_1\ldots x_5 [N_0] = -\prod_{i=1}^5c_i.H \int_{N_0}\Theta^5 = -5\cdot 2^4\prod_{i=1}^5c_i.H. \]
Hence, $\deg_{u_1}N_0 = 5 \cdot 2^9$.
\end{prop}
\begin{proof}
Let $i \colon C \hookrightarrow S$ be the inclusion. The inclusion $N_0 \hookrightarrow M$ is defined by $(\id\times i)_*\Ecal_{\rm univ}$, where $\Ecal_{\rm univ}$ is the universal vector bundle on $N_0 \times C$. Hence,
\[ x\!\left|_{N_0} \right. = \lambda_{(\id\times i)_*\Ecal_{\rm univ}}(u) = \lambda_{N_0}(Li^*u). \]
Now,
$Li^* \colon K(S)_{\rm num} \rightarrow K(C)_{\rm num} \cong \Z^{\oplus 2}$
maps $(r,c,s)$ to $(r,c.H)$. In particular, $Li^*u  = c.H(2,1)$, whereas by definition
$\theta = \lambda_{N_0}(-2,-1).$

Next, we compute $\int_{N_0}\Theta^5$ by pulling back along $h \colon \SM_C(2,1) \times \Pic^0(C) \rightarrow N_0$ from Remark \ref{pullback of theta}. 
\begin{align*}
\int_{\M_C(2,1)} \Theta^5 &\stackrel{\eqref{formula for theta}}{=} \frac{1}{2^4}\int_{\SM_C(2,1)\times \Pic^0(C)}(p_1^*\Theta_{\SM} + 4p_2^*\Theta_0)^5 \\
 &=\frac{1}{2^4}\binom{5}{3}\int_{\SM_C(2,1)}\Theta_{\SM}^3\int_{\Pic^0(C)}(4\Theta_0)^2= 5\cdot 2^4.
\end{align*}

The value $\int_{\SM_C(2,1)}\Theta_{\SM}^3 = 4$ is given by the leading term of the Verlinde formula \cite{Zag2}.
\end{proof}

\begin{rem}
The general formula is
\[ \int_{\M_C(n,d)}\Theta^{\dim\M_C(n,d)} = \dim\M_C (n,d)!(2^{2g-2}-2)\frac{(-1)^g2^{2g-2}B_{2g-2}}{(2g-2)!}, \]
where $B_i$ is the $i$-th Bernoulli number. The second Bernoulli number is $B_2= \frac{1}{6}$
\end{rem}

\begin{rem}
In the general case, where $v = (0,nH,s)$ and $u_1 = \lambda_M(-n(2g-2),sH, *)$ with $s=n+d(1-g)$, we find $u_1\!\left|_{F}\right. = n(2g-2)\Theta_\delta$ and $u_1\!\left|_{N_0}\right. = (2g-2)\Theta$. Thus
\begin{equation*}
\begin{array}{lcr}
\deg_{u_1}F= (n(2g-2))^{\dim N}\cdot \dim N! & \text{and} & \deg_{u_1}N_0 = (2g-2)^{\dim N}\int_{\M_C(n,d)}\Theta^{\dim\M_C(n,d)}.
\end{array}
\end{equation*}
Here, $\dim N = n^2(2g-2)+2$.
\end{rem}

\subsection{Degree of the other component \texorpdfstring{$N_1$}{}}
We complete the proof of Theorem \ref{Thm1} by dealing with the remaining component $N_1$.
Recall from Proposition \ref{nu is there and birational} that there is a birational map
$\nu \colon \overline{W} \rightarrow N_1,$
where $ \tau \colon \overline{W} = \Pbb(\Wcal) \rightarrow T = {\Pic^1(C)} \times C$.
\begin{prop}\label{other comp}
Let $x_i= \lambda_{M}(u_i)$ with $u_i=(2c_i.H,c_i,s_i) \in v^\bot$ for $i=1,\ldots,5$. Then

\begin{equation}\label{temp73}
\int_M x_1\ldots x_5 [N_1]  = \int_{\overline{W}}\prod_{i=1}^5\nu^*(x_i\!\left|_{N_1}\right.) = -5^2\cdot 2^9\prod_{i=1}^5c_i.H.
\end{equation}
In particular, $\deg_{u_1}N_1 = 5^2 \cdot 2^{11}$.
\end{prop}

Note that the first equality in \eqref{temp73} is immediate, because $\nu \colon \overline{W} \rightarrow N_1$ is birational. For the proof of the proposition, let us introduce some more notation. We abbreviate $\Pic^1(C)$ to $\Pic^1$ and in the following all cohomology groups have $\Z$ coefficients. We set
\begin{equation*}
\begin{array}{lcr}
\zeta = c_1(\Ocal_\tau(1)) \in H^2(\overline{W}) &\text{and write}& \rho = p_2^*[\rm{pt}] \in H^2(\Pic^1 \times C)
\end{array}
\end{equation*}
for the pullback of the class of a point on $C$. If no confusion is likely, we suppress pullbacks from our notation, e.g.~we will write $\Theta_1 \in H^2(\Pic^1 \times C)$ and also $\Theta_1 \in H^2(\Pbb(\Wcal))$ instead of $p_1^*\Theta_1$ and $\tau^*p_1^*\Theta_1$, respectively. 
Moreover, we define
\[ \pi \coloneqq c_1(\Pcal) -c_1^{2,0}(\Pcal) \in H^2(\Pic^1 \times C),\]
where $\Pcal$ is a Poincar\'e line bundle. Note that $\pi$ is independent of the choice of $\Pcal$.

\begin{proof}[Proof of Proposition \ref{other comp}]
We will split the proof into the following three steps.
\begin{enumerate}[\rm (i)]
\item Let $x = \lambda_{M}(2c.H,c,s)$. Then
\[ \nu^*(x\!\left|_{N_1}\right.) = \lambda_{\Gcal_{\rm univ}}(x) = c.H(-4\Theta_1+ 2\pi -7 \rho -\zeta) \in H^2(\overline{W}).\]
\item We have
\[(-4\Theta_1+ 2\pi -7 \rho -\zeta)^5 = -5^2 2^{5}\zeta^2\rho\Theta_1^2 \in H^{10}(\overline{W}).\]
\item The top cohomology group $H^{10}(\overline{W})$ generated by $\tfrac{1}{2}\zeta^2\rho\Theta_1^2$ and we have
\[\int_{\Pbb(\Ecal)}\zeta^2\rho\Theta_1^2 = 2.\]
\end{enumerate}

\begin{proof}[Proof of (i)]
The morphism $\nu \colon \overline{W} \rightarrow N_1$ corresponds to $\Gcal_{\rm univ} \in \Coh(\overline{W} \times S)$, which sits in the (universal) extension
\[  0 \rightarrow \tau_S^*(\id\times i)_*(\Pcal_1 \boxtimes \Ocal(\Delta) \boxtimes \omega_C^{-1} )\boxtimes \Ocal_\tau(1) \rightarrow \Gcal_{\rm univ} \rightarrow \tau_S^*(\id\times i)_*p_{\sss 13}^*\Pcal_1 \rightarrow 0,\]
where $\tau_S = \tau \times \id_S \colon \overline{W} \times S \rightarrow \Pic^1 \times C \times S$. So, by construction, we have
\begin{align*}
\lambda_{\Gcal_{\rm univ}}(x) &= \lambda_{\tau_S^*(\id\times i)_*(\Pcal_1 \boxtimes \Ocal(\Delta) \boxtimes \omega_C^{-1}) \boxtimes \Ocal_\tau(1)}(x) + \lambda_{\tau_S^*(\id\times i)_*p_{\sss 13}^*\Pcal_1}(x)\\
& = \lambda_{\tau_S^*(\id\times i)_*(\Pcal_1 \boxtimes \Ocal(\Delta) \boxtimes \omega_C^{-1})}(x) + k(Li^*x)\cdot\zeta + \tau^*p_1^*\lambda_{\Pcal_1}(Li^*x)\\
& = \tau^*(\lambda_{\Pcal_1 \boxtimes \Ocal(\Delta)}(Li^*x\cdot\omega^{-1})+ p_1^*\lambda_{\Pcal_1}(Li^*x))+ k(Li^*x)\cdot\zeta,\\
\end{align*}
where $\omega = c_1(\omega_C)$ and $k(Li^*x) = \rk Rp_* (\Pcal_1 \boxtimes \Ocal(\Delta)\boxtimes \omega_C^ {-1} \boxtimes Li^*x) = \chi(Li^*x ) = - c.H$.

The term $\lambda_{\Pcal_1 \boxtimes \Ocal(\Delta)}(Li^*x\cdot\omega^{-1})+ p_1^*\lambda_{\Pcal_1}(Li^*x)$, is determined in Lemmas \ref{Hilfslemma 1} and \ref{lambda poincare}. Note that each summand depends on the choice of a Poincar\'e line bundle, whereas the sum does not. Together,
\begin{align*}
\nu^*(x\!\left|_{N_1}\right.) &= \tau^*(\lambda_{\Pcal_1 \boxtimes \Ocal(\Delta)}(Li^*x\cdot\omega^{-1})+ p_1^*\lambda_{\Pcal_1}(Li^*x))-c.H\zeta \\
&= c.H (p_1^*(\lambda_{\Pcal_1}(2,-3)+ \lambda_{\Pcal_1}(2,1)) + 2c_1(\Pcal_1) -7\rho -\zeta)\\
& = c.H ( -4\Theta_1 + 2\pi-7\rho -\zeta).
\end{align*}
\end{proof}

\begin{lem}\label{Hilfslemma 1}
Let $\Fcal \in \Coh(X \times C)$. Then
\[ \lambda_{\Fcal \boxtimes \Ocal(\Delta)} (x) = {p_1}^*\lambda_{\Fcal}(x) + rc_1(\Fcal) + c_0(\Fcal)(d-2r)\rho \]
for all $x = (r,d) \in \K(C)_{\rm num}$. In particular,
\[ \lambda_{\Pcal_1 \boxtimes \Ocal(\Delta)}(Li^*x\cdot\omega^{-1})
= c.H (p_1^*\lambda_{\Pcal}(2,-3) + 2c_1(\Pcal_1) -7\rho). \]
\end{lem}

\begin{proof}
We have 
\[ [\Fcal \boxtimes \Ocal(\Delta)] = [p_{13}^*\Fcal] + [(\id\times i_\Delta)_*(p_2^*\omega_C^{-1} \otimes \Fcal)] \in \K (X\times C\times C), \]
where $i_\Delta\colon C \rightarrow C \times C$ is the diagonal and thus
\[ \lambda_{\Fcal \boxtimes \Ocal(\Delta)}(x) = \lambda_{p_{13}^*\Wcal}(x) + \lambda_{(\id\times i_\Delta)_*(p_2^*\omega_C^{-1} \otimes \Fcal)}(x) \in H^2(X \times C) \]
for all $x \in \K (C)_{\rm num}$. Now,
\begin{align*}
\lambda_{(\id\times i_\Delta)_*(p_2^*\omega^{-1} \cdot [\Fcal])}(x) &= \det R{p_{12}}_*((\id\times i_\Delta)_*(p_2^*\omega^{-1} \cdot [\Fcal]) \cdot p_3^*x)\\
&= \det R(p_{12}\circ (\id \times i_\Delta))_*([\Fcal]\cdot p_2^*(\omega^{-1} \cdot x))\\
&= \det([\Fcal] \cdot p_2^*(\omega^{-1} \cdot x)) = rc_1(\Fcal) + c_0(\Fcal)(d-r(2g-2))\rho.
\end{align*}
\end{proof}

To prove the remaining steps, we need to understand the cohomology ring $H^*(\overline{W})$. 
\begin{lem}
We have
\[H^*(\overline{W})\cong H^*(\Pic^1 \times C)[\zeta]/\zeta^3+4\rho\zeta^2.\]
In particular,
\[ H^{10}(\overline{W}) = \zeta^2 \cdot H^6(\Pic^1 \times C). \]
\end{lem}
\begin{proof}
By definition, $\overline{W} = \Pbb(\Wcal)$. Hence,
\[H^*(\overline{W})\cong H^*({\Pic^1} \times C)[\zeta]/\zeta^3+c_1(\Wcal)\zeta^2+ c_2(\Wcal)\zeta+ c_3(\Wcal).\]
We use the short exact sequence $0 \rightarrow \Vcal \rightarrow \Wcal \rightarrow \Ocal_T \rightarrow 0$ from \eqref{ses} to compute the Chern classes of $\Wcal$. Note that
$\Vcal = R^1{p_{\sss 12}}_*(p_{\sss 23}^*\Ocal(\Delta)\otimes p_3^*\omega_C^{-1}) \cong p_2^*R^1 {p_1}_*(\Ocal(\Delta)\otimes p_2^*\omega_C^{-1}).$
So the Chern classes of $\Vcal$ can be computed by the push forward along the first projection of the following the short exact sequence 
$ 0 \rightarrow p_2^*\omega_C^{-1} \rightarrow  \Ocal(\Delta) \boxtimes \omega_C^{-1} \rightarrow p_2^*\omega_C^{- 2}\!\left|_\Delta\right. \rightarrow 0.$
We find
\[ 0 \longrightarrow \omega_C^{-2} \rightarrow \Ocal_C \otimes H^1(C,\omega_C^{-1}) \rightarrow R^1{p_1}_*(\Ocal(\Delta)\otimes p_2^*\omega_C^{-1}) \longrightarrow 0. \]
Hence,
 $c_1(\Wcal) = 4 \rho$ and $c_i(\Wcal) = 0$ if $i \geq 2$.
\end{proof}

\begin{proof}[Proof of (ii) and (iii)]
We want to show that
\[ (-4\Theta_1+ 2\pi -7 \rho -\zeta)^5 = -5^2 2^{5}\zeta^2\rho\Theta_1^2 \in H^{10}(\overline{W}).\]
We compute
\begin{align*}
(-4\Theta_1+2\pi -7 \rho -\zeta)^5& = \binom{5}{3}(-\zeta^3)(-4\Theta_1+2\pi -7 \rho)^2 + \binom{5}{2}\zeta^2(-4\Theta_1+2\pi -7 \rho)^3\\
& =10\cdot \zeta^2 ((4\rho(-4\Theta_1+2\pi -7 \rho)^2 + (-4\Theta_1+2\pi -7 \rho)^3).
\end{align*}
The result is a combination of $\pi,\theta$ and $\rho$, which are classes of type $(1,1)+(0,2), (2,0)$ and $(0,2)$, respectively. Moreover, in the proof of Lemma \ref{lambda poincare} we computed
$\pi = \rho + \gamma$ and $\pi^2 = \gamma^2 = -2\rho\Theta_1.$ Hence, the only non-zero combinations are
$\pi^2\Theta_1 = -2\rho \Theta_1^2 = -2 \pi\Theta_1^2.$
We find
\begin{multline*}
10\cdot \zeta^2 ((4\rho(-4\Theta_1+2\pi -7 \rho)^2 + (-4\Theta_1+2\pi -7 \rho)^3)\\
 = 10\cdot \zeta^2 (2^6\rho\Theta_1^2 + 3(-2^4\pi^2\Theta_1 + 2^5\pi\Theta_1^2 - 7\cdot 2^4\rho\Theta_1^2)\\
 = 10  (2^6+ 3(2^5 +2^5 - 7\cdot 2^4))\zeta^2\rho\Theta_1^2 = -5^2  2^{5} \zeta^2\rho\Theta_1^2.
\end{multline*}

Finally, we want to show that $\int_{\overline{W}}\zeta^2\rho\Theta^2 = 2$. Indeed,
\[\int_{\overline{W}}\zeta^2\rho\Theta^2 
=\tau_*\zeta^2 \int_{\Pic^1}\Theta^2\int_{C}\rho =2.\]

\end{proof}
This concludes the proof of the proposition.
\end{proof}
	
\section{Proof of Theorem \ref{Thm2}}\label{section classes}
In this section, we prove Theorem \ref{Thm2}, i.e.\ the following formulae
\[ [N_0] = \frac{1}{48} [F] + \beta\ \ \text{and}\  \ [N_1] = \frac{5}{12} [F] - 4 \beta,
\]
for the the classes $[N_0]$ and $[N_1] \in H^{10}(M,\Q)$. Here, $[F]$ is the class of a general fiber and $0 \neq \beta \in (S^5H^2(M),\Q)^\bot$ satisfies $\beta^2 = 0$.
From now on, all cohomology groups have $\Q$-coefficients.\\

Before coming to the proof, we want to point out, that the irreducible components over points in $\Sigma \setminus \Delta$ (see \eqref{fibertype}) are of different cohomological nature. Let $D \in \Sigma\setminus \Delta$ be a reducible curve with two smooth components $C_1$ and $C_2$. Then the two components $N'_1$ and $N'_2$ of $f^{-1}(D)$ contain an open sublocus parametrizing line bundles on $D$ of bi-degree $(2,1)$ and $(1,2)$, respectively \cite[Proposition 3.7.1 and Lemma 3.3.2]{CRS}. The monodromy around $\Sigma\setminus\Delta$ exchanges $C_1$ and $C_2$ and consequently the classes of the irreducible components.  We find
\[ [N'_1] = [N'_2] = \tfrac{1}{2}[F]. \]
In particular, the two components are linearly dependent. This is not true over $\Delta$.
\begin{prop}\label{linearly independent}
The classes $[N_0]$ and $[N_1] \in H^{10}(M)$ are linearly independent.
\end{prop}
The proof uses the following simple observation.
\begin{lem}
Let $M \rightarrow B$ be a Lagrangian fibration and $F$ a smooth fiber. Then
\[ c_i(\Tcal_M)\!\left|_F\right. = 0\ \text{for all}\ i>0. \]
\end{lem}
\begin{proof}
We have a short exact sequence
$0 \rightarrow \Tcal_F \longrightarrow \Tcal_M\!\left|_F\right. \longrightarrow \Ncal_{F/M} \rightarrow 0.$
Now, $F \subset M$ is Lagrangian and hence $\Ncal_{F/M} \cong \Omega_F$. Moreover, $F$ is an abelian variety and hence all its Chern classes of degree greater than zero are trivial. 
\end{proof}

\begin{proof}[Proof of Proposition \ref{linearly independent}]
Assume that $[N_0]$ and $[N_1]$ are linearly dependent. Then there is some $\lambda \in \Q$ such that $[F] = \lambda [N_0]$, where $F \subset M$ is a smooth fiber. In particular, by the above lemma, any product of $[N_0]$ and the Chern classes of $M$ vanishes. However, we will show that
\[\int_M c_2(\Tcal_M)\cdot u_1^3\cdot[N_0] \neq 0,\]
leading to the desired contradiction. We have $c(\Tcal_M\!\left|_{N_0}\right.) = c(\Tcal_{N_0})c(\Omega_{N_0})$ and thus
\[ c_2(\Tcal_M)\!\left|_{N_0}\right. = (2c_2 -c_1^2)(\Tcal_{N_0}). \]
This gives,
\begin{align*}
\int c_2(\Tcal_M)u_1^3[N_0] &{\stackrel{\ref{vector bundle comp}}{=}} \int_{N_0}(2c_2-c_1^2)(\Tcal_{N_0})\cdot (2\Theta)^3\\
&=\frac{1}{2^4}\int_{\SM_C(2,1)\times \Pic^0}h^*((2c_2-c_1^2)(\Tcal_{N_0})\cdot (2\Theta)^3)\\
&{\stackrel{\eqref{formula for theta}}{=}} \frac{2^3}{2^4}\int_{\SM_C(2,1)\times \Pic^0}p_1^*(2c_2-c_1^2)(\Tcal_{\SM})\cdot(p_1^* \Theta_{\SM}+ 4 p_2^*\Theta_{0})^3\\
&= \frac{1}{2} \int_{\SM_C(2,1)\times \Pic^0} p_1^*(2c_2-c_1^2)(\Tcal_{\SM})\cdot(3p_1^* \Theta_{\SM}\cdot 4^2 p_2^*\Theta_{0}^2)\\
&= 3\cdot2^3 \int_{\SM_C(2,1)}(2c_2-c_1^2)(\Tcal_{\SM})\cdot \frac{1}{2}c_1(\Tcal_{\SM}) \int_{\Pic^0}\Theta_{0}^2\\
&= 3 \cdot 2^4 \int_{\SM_C(2,1)}(6\alpha^2-4\alpha^2)\alpha = 3\cdot 2^7 \neq 0,
\end{align*}
where, in the last line, $\alpha \in H^2(\SM_C(2,1))$ is the degree two K\"unneth component of $(c_1^2-c_2)(\Vcal_{\rm univ})$ with $\Vcal_{\rm univ}$ being a universal bundle on $\SM_C(2,1)\times C$. It is known, e.g.\ \cite[\S 5A]{Zag}, that
\begin{equation*}
\begin{array}{lcr}
c_1(\Tcal_{\SM_C(2,1)}) = 2\alpha & \text{and}  & c_2(\Tcal_{\SM_C(2,1)}) = 3\alpha^2
\end{array}
\end{equation*}
as well as $\int\alpha^3 = 4$. We also used from \cite[Th\'eor\`eme F]{DN} that $\Ocal_{\SM(2,1)}(\Theta)^{-2} \cong \omega_{\SM(2,1)}$. Hence,
\[ \Theta = - \frac{1}{2} c_1(\omega_{\SM(2,1)}) = \frac{1}{2}c_1(\Tcal_{\SM_C(2,1)}).\]
\end{proof}

\begin{proof}[Proof of Theorem \ref{Thm2}]
We set $V \coloneqq S^5H^2(M)\subset H^{10}(M)$ so that we have an orthogonal decomposition with respect to the cup product
$ H^{10}(M) = V \oplus V^\bot.$
Accordingly, we write $[N_i] = \alpha_i + \beta_i$ with $\alpha_i \in V$ and $0 \neq \beta_i \in V^\bot$ for $i=1,2$.
We claim that
\begin{equation}\label{temp23}
20 [N_0] - [N_1]\in V^\bot.
\end{equation}
To see this, we decompose the second cohomology group into its transcendental and algebraic part, i.e.\ $H^2(M) = T(M) \oplus \NS (M)$. Now, for $i=1,2$ consider
\begin{equation}\label{temp40}
T(M) \rightarrow H^{12}(M),\ \ \alpha \mapsto \alpha\cdot[N_i].
\end{equation}
As the symplectic form $\sigma \in T(M)$ vanishes on $N_i$, it follows by irreducibility of the Hodge structure $T(M)$ that the assignment \eqref{temp40} is trivial. Hence, it suffices to show that $20 [N_0] - [N_1] \in (S^5\NS (M))^\bot$. By \eqref{Mukai homo} any element in $S^5\NS(M)$ is of the form $x_1x_2\ldots x_5$, where
$x_i = \lambda_{M}(2c_i.H,c_i,s_i)$. According to Propositions \ref{vector bundle comp} and \ref{other comp}
\[ \int[N_1]x_1x_2\ldots x_5 = -5^2 2^6\prod_{i=1}^5 c_i.H = 20  \int[N_0]x_1x_2\ldots x_5.\]
This proves \eqref{temp23}.

Next, we write $[N_1] - 20 [N_0] = \alpha_1 -20 \alpha_0 + \beta_1 - 20 \beta_0 \in V^\bot$ and conclude $\alpha_1 = 20 \alpha_0$. We set $\alpha = \alpha_0$. From the multiplicities found in Theorem \ref{Thm1} we have
\[ 2^3 [N_0] + 2[N_1] = [F] = u_0^5 \in V\]
and also $ u_0^5 = 48\alpha + 8 \beta_0 + 2\beta_1$. This gives $48\alpha = u_0^5$ and $\beta_1 = -4 \beta_0$. Setting $\beta = \beta_0$ gives the desired expression.

The last assertion follows from
$ [N_0]^2 = (\tfrac{1}{48}u_0^5 + \beta)^2 = \beta^2$
and
\[ [N_0]^2= \int_{N_0}c_5(\Ncal_{N_0/M}) = \int_{N_0}c_5(\Omega_{N_0}) = - e(N_0),\]
which is known to vanish, see \cite[\S 9]{AB}. Hence also $[N_1]^2=0$ and $[N_i]\cdot \beta = 0$ for $i=1,2$.
\end{proof}

\end{document}